\documentclass{amsart}
\usepackage{amsmath}
\usepackage{graphicx}
\usepackage{latexsym}

\newtheorem{thm}{Theorem}
\newtheorem{lem}[thm]{Lemma}
\newtheorem{cor}[thm]{Corollary}
\newtheorem{prop}[thm]{Proposition}

\newtheorem{rk}[thm]{Remark}

\newtheorem{problem}[thm]{Problem}

\newcommand{\CPb}{\overline{\mathbb{CP}}{}^{2}}
\newcommand{\CP}{{\mathbb{CP}}{}^{2}}

\newcommand{\K}{\textbf{K3}}
\newcommand{\Kb}{\overline{\textbf{K3}}}

\newcommand{\Z}{\mathbb{Z}}

\def \x {\times}
\def \eu{{\text{e}}}

\begin{document}

\title[Broken Lefschetz fibrations and smooth structures on $4$-manifolds] {Broken Lefschetz fibrations and \\ smooth structures on $4$-manifolds}
\vspace{0.2in} 

\author[R. \.{I}. Baykur]{R. \.{I}nan\c{c} Baykur}
\address{Max Planck Institute for Mathematics, Bonn, Germany \newline
\indent Department of Mathematics, Brandeis University, Waltham MA, USA}
\email{baykur@mpim-bonn.mpg.de, baykur@brandeis.edu}

\begin{abstract} The broken genera are orientation preserving diffeomorphism invariants of closed oriented $4$-manifolds, defined via broken Lefschetz fibrations. We study the properties of the broken genera invariants, and calculate them for various $4$-manifolds, while showing that the invariants are sensitive to exotic smooth structures.
\end{abstract}

\maketitle

\vspace{-0.2in} 

\section{Introduction}

Broken Lefschetz fibrations are generalizations of Lefschetz fibrations on smooth $4$-manifolds, which are allowed to have indefinite fold singularities along embedded circles in addition to Lefschetz type singularities on a discrete set. In the recent past, there has been a flurry of activity around broken Lefschetz fibrations, extending ideas stemmed in symplectic geometry and gauge theory on one end by Auroux-Donaldson-Katzarkov, and Perutz \cite{ADK,P}, and employing handlebody and singularity theories to suggest new ways to study the topology of $4$-manifolds on the other end by the author, Saeki, Gay-Kirby, Lekili, Akbulut-Karakurt, Williams, Hayano, and others \cite{AK, B1, B2, B3, BK, SB, C, H1, H2, HS1, HS2, GK1, GK2, L, Sa, W}. Given any surjective map from a closed oriented $4$-manifold $X$ to the $2$-sphere, there exists a rather special broken Lefschetz fibration on $X$ within the same homotopy class, with only connected fibers and no exceptional spheres contained on the fibers, with at most one circle of indefinite fold singularities whose image in the base is embedded, and where all the Lefschetz critical points lie on fibers with the highest genera. These are called \emph{simplified broken Lefschetz fibrations} (SBLF in short), and constitute an important subfamily of broken Lefschetz fibrations, allowing one to study the underlying topology effectively. (See for instance \cite{B1, B2, B3, BK, H1, H2, HS1, HS2}.) The underlying topology of a simplified broken Lefschetz fibrations is rather simple: It is either a relatively minimal genus $g$ Lefschetz fibration over the $2$-sphere, or it decomposes as a relatively minimal genus $g$ Lefschetz fibration over a $2$-disk, a trivial genus $g-1$ bundle over a $2$-disk, and a fibered cobordism in between prescribed by a single round handle \cite{B1}.

Given a closed smooth oriented $4$-manifold $X$ and a homology class $\alpha \in H_2(X; \Z)$ with $\alpha^2 = m \geq 0$, let $\tilde{X} \cong X\#m\CPb$ be the blow-up of $X$ at $m$ points and $\tilde{\alpha}= {\iota}_*(\alpha)- \sum_i e_i$, where $\iota$ is the inclusion into $\tilde{X}$ of the complement of the blown-up points in $X$ and $e_i$ are the homology classes of the exceptional spheres. The \emph{broken genus} of $(X, \alpha)$ is the smallest non-negative integer $g$ among the genera of all possible SBLFs on $\tilde{X}$ whose fibers represent the class $\tilde{\alpha}$, which we denote by $g(X, \alpha)$. The \emph{broken genus series} of $X$, denoted by $S_g(X)$, is a formal sum in the integral group ring $\Z[H_2(X)]$, defined as 
\[ 
S_g(X)  \, = \sum_{\alpha \in H_2(X)\, , \, \alpha^2=0} \, g(X, \alpha) \, t_{\alpha} \, \, \,  + \sum_{\beta \in H_2(X)\, , \, \beta^2 >0} \, g(X, \beta) \, t_{\beta} \, .
\]
We then define the \emph{broken genus} of $X$, denoted by $g(X)$, as the smallest coefficient that appears in the first series on the right hand side. As shown in the next section, the broken genus of such a pair $(X, \alpha)$, the broken genus series, and the broken genus of $X$ are all orientation preserving diffeomorphism invariants. 

The purpose of this article is two-fold: We aim to initiate a study of smooth structures on $4$-manifolds (and smooth embeddings of surfaces) via broken genus invariants. Our second goal is to suggest a systematic framework for the rapidly growing literature on the topology of broken Lefschetz fibrations on $4$-manifolds by considering SBLFs of fixed genus. We are going to use both handlebody and singularity theory arguments at times, whichever comes handy depending on the context. Clearly, in either case we essentially work with maps with prescribed singularities, i.e Morse or indefinite generic singularities. The freedom of choice we give ourselves to switch between the two approaches rests in the following distinction: The former approach has the advantage of explicitly prescribing the fibration \textit{and} the topology of the underlying $4$-manifold, whereas the latter approach is more useful for implicit modifications of \textit{given} fibrations whenever one only needs to keep track of some rough information such as the genus of the fibration. In particular, simplified purely wrinkled fibrations (SPWF in short) studied by Williams \cite{W}, which have similar topological properties as SBLFs except for having cusps instead of Lefschetz singularities, appear in several intermediary steps of our arguments. Since one can turn SBLFs into SPWFs in a straightforward fashion, all of our work here can be rephrased in terms of simplified purely wrinkled fibrations easily. 

The organization of the article and a summary of the results are as follows: 

In Section \ref{Preliminaries}, we review the well-definedness of the broken genera invariants and give various preparatory results. In Sections \ref{Properties} and \ref{Exotic}, we present some calculational tools and provide many calculations of the broken genera of $4$-manifolds, both using handlebody and singularity theory arguments. We get lower and upper bounds on the broken genera, which vary from purely topological (i.e. based solely on the homeomorphism type) bounds (Lemmas \ref{euler} and \ref{pi1}) to smooth ones (Propositions \ref{Thom} and \ref{orientation}, Theorem \ref{SPWFconnected}). Here we show that the broken genera invariants can attain arbitrarily large values (Proposition~\ref{unbounded}), even on pairs in a fixed homeomorphism class (Theorem~\ref{knotsurgery}). 

In Section \ref{Exotic}, we show how broken genera invariants are sensitive to exotic smooth structures on various $4$-manifolds, most remarkably, on standard simply connected $4$-manifolds. Theorem~\ref{constructions} calculates the broken genera of standard simply-connected $4$-manifold blocks, and in particular, presents many examples of SBLFs. Theorem~\ref{SW} and Corollary~\ref{symplectic} (and also Theorem~\ref{knotsurgery}) demonstrate that the broken genera of exotic simply-connected $4$-manifolds are higher, relying on Freedman's topological classification of simply-connected $4$-manifolds and on Seiberg-Witten theory. The results contained in this section, and more specifically the construction of genus one simplified broken Lefschetz fibrations have been promised for some time
and were presented by the author during the Four Dimensional Topology Conference in Hiroshima in 2010. 

It has been asked by various authors (\cite{P, B0, W}) whether or not this sort of invariants encode information regarding the minimal genus representative of a \linebreak $2$-dimensional homology class on a $4$-manifold. Section~\ref{BrokenVersusMinimal} contains a discussion on this topic, where we focus on showing how broken genus relates to and differ from minimal genus. The final section lists some related open problems.

\newpage

\section{Preliminaries and background results} \label{Preliminaries}

Let $X$ and $\Sigma \, $ be compact connected oriented manifolds (with or without boundary) of dimension four and two, respectively, and $f\colon\,  X \to \Sigma \, $ be a smooth surjective map with $f^{-1}(\partial \Sigma) = \partial X$. The map $f$ is said to have a Lefschetz singularity at a point $x$ contained in a discrete set $C=C_f \subset Int(X)$, if around $x$ and $f(x)$ one can choose orientation preserving charts so that $f$ conforms the complex local model 
\[(u, v) \to u^2 + v^2 \,. \]
The map $f$ is said to have an \emph{indefinite fold} or \textit{round singularity} along an embedded $1$-manifold $Z=Z_f \subset Int(X) \setminus \, C$\, if around every $z \in Z$, there are coordinates $(t, x_1, x_2, x_3)$ with $t$ a local coordinate on $Z$, in terms of which $f$ can be written as 
\[(t, x_1, x_2, x_3) \to (t, x_1^2 - x_2^2 - x_3^2) \,.\]
We call the image $f(Z) \subset Int(\Sigma)$ the \textit{round image}. A \textit{broken Lefschetz fibration} is then defined as a smooth surjective map $f\colon\,  X \to \Sigma \, $ which is submersion everywhere except for a finite set of points $C$ and a finite collection of circles $Z \subset X \setminus C$, where it has Lefschetz singularities and round singularities, respectively. In particular, it is an honest surface bundle over $\partial \Sigma$.

Let us also recall the simplest types of singularities for smooth maps. Let $y \in Int(X)$ be a point where $\text{rank}(d f_y) < 2$. The map $f\colon\, X\to \Sigma$ is said to have a \emph{fold singularity} at a point $y \in Int(X)$, if around $y$ and its image one can choose local coordinates so that the map is given by
\[(t, x_1, x_2, x_3) \mapsto (t, \pm x_1^2 \pm x_2^2 \pm x_3^2) ,\]
and a \emph{cusp singularity} if instead the map is locally given by
\[(t, x_1, x_2, x_3) \mapsto (t, x_1^3 + t x_1 \pm x_2^2 \pm x_3^2) . \]
We say that $y$ is a \emph{definite fold point} if all the coefficients of $x_i^2$, $i=1,2,3$ in the first local model is of the same sign. It is called an \emph{indefinite fold point} otherwise. From a special case of Thom's transversality theorem it follows that any smooth map from an $n$-dimensional manifold to a $2$-manifold, for $n \geq 2$, can be approximated arbitrarily well by a map with only fold and cusp singularities \cite{Lev0}. A smooth map $f\colon\, X \to \Sigma$ with only fold and cusp singularities is called a \emph{generic map}. The singular set of $f$ is a finite set of circles, which are composed of finitely many cusp points and a finite collection of arcs and circles of fold singularities. From the very definitions we see that round singularities are circles of indefinite fold singularities. As the set of generic maps is open and dense in $C^{\infty}(X, \Sigma)$ topologized with the $C^{\infty}$ topology, any broken Lefschetz fibration can be homotoped to a map with only fold and cusp singularities. 

As shown in \cite{Sa, B2}, any generic map from a closed orientable $4$-manifold to the $2$-sphere can be homotoped to a broken Lefschetz fibrations over $\Sigma = S^2$, and thus, these fibrations are found in abundance. Moreover, any given framed surface $F$ with trivial normal bundle in $X$ can be chosen to be a fiber of such a fibration \cite{B2, GK1}. We will say that $(X, \alpha)$ \textit{admits a broken Lefschetz fibration} whenever there is a broken Lefschetz fibration on $X$ whose fiber is in the homology class of $\alpha$. Note that $\alpha^2=0$ is therefore a necessary condition.

Whenever there is a fiber in $X$ containing a self-intersection $-1$ sphere, it can be blown-down to obtain a new broken Lefschetz fibration on $X'$ without destroying the fibration structure on the rest. We will therefore focus on \textit{relatively minimal broken Lefschetz fibrations} (possibly on \textit{non-minimal} $4$-manifolds, which do not contain such fiber components.

\textit{Simplified broken Lefschetz fibrations} make up a subfamily of broken Lefschetz fibrations subject to the following additional conditions: The base surface of the fibration is $S^2$, the round image is connected (possibly empty) and the round image is embedded, the fibration is relatively minimal, and whenever $Z \neq \emptyset$, all the fibers are connected and all the Lefschetz singularities lie over the $2$-disk component of $S^2 \setminus f(Z)$ over which fibers have higher genera. As we will see below, in any homotopy class of a surjective map from a given closed oriented $4$-manifold $X$ to $S^2$, there exists a simplified broken Lefschetz fibration. The fibration is composed of three pieces when the round locus is non-empty: If we take an annular neighborhood $A$ of the round image avoiding the Lefschetz singularities, then decomposing the base as $S^2 = D^2 \cup A \cup D^2$ we get a genus $g-1$ surface bundle over $D^2$, called the \textit{lower side}, a genus $g$ Lefschetz fibration over $D^2$, called the \textit{higher side}, and a round cobordism in between composed of only one round handle. It is then easy to give a handlebody decomposition of $X$ prescribing this simplified broken Lefschetz fibration, as shown in \cite{B1}. The genus $g$ of the higher side of an SBLF $f\colon\,  X \to S^2$ is called the \textit{genus of} $f$.\footnote{A word of caution here: We define the genus of the SBLF as the genus of the \emph{higher side} fiber, which is in agreement with \cite{BK, H1, H2, HS1, HS2} and is equal to the genus of an honest Lefschetz fibration. The same goes for SPWFs. Unlike in \cite{W}, our broken genus invariant picks up its value among these higher side genera.} Similarly, \textit{simplified purely wrinkled fibrations} make up a subfamily of indefinite generic maps mimicking the conditions we had for SBLFs: The base surface is $S^2$, the singular locus is a connected $1$-manifold with embedded image which is smooth everywhere except at the images of a \textit{non-empty} set of cusp points, and all the fibers are connected. These fibrations were introduced and studied by Williams in \cite{W}. 

There are various topological modifications of broken Lefschetz fibrations and indefinite generic maps defined via handlebody operations or homotopy moves between generic maps without definite folds. For the details of these moves we will employ in this article, the author is advised to turn to \cite{B1, B2} and \cite{L, W}.  

We start with a few observations:

\begin{lem} \label{push}
If $(X, \alpha)$ admits a genus $g$ broken Lefschetz fibration with non-empty round locus, then $(X, \alpha)$ admits a homotopic genus $g$ broken Lefschetz fibration with all the Lefschetz singularities on the higher side.  
\end{lem}

\begin{proof}
The proof of this was given in \cite{B1} for near-symplectic broken Lefschetz fibrations, that is, for broken Lefschetz fibrations with the additional property that there is a second cohomology class evaluating non-trivially on every fiber component. The same proof can be adapted mutatis mutandis to this general setup. 
\end{proof}

\begin{lem} \label{relativeminimality}
If $(X, \alpha)$ admits a genus $g$ SBLF with non-empty round locus, then $(X\#\CPb, \tilde{\alpha})$ where the class $\tilde{\alpha}$ represents the image of $\alpha$ under the homomorphism induced by the inclusion map $X \setminus D^4 \hookrightarrow \tilde{X}$, also admits a genus $g$ SBLF. If $(X, \alpha)$ admits a genus $g$ broken Lefschetz fibration with non-empty round locus, meeting all the additional conditions listed for SBLFs but the relative minimality, then $(X, \alpha)$ admits a genus $g$ SBLF as well. 
\end{lem}

\begin{proof}
The first claim is best seen from the handle decomposition of the SBLF on $X$ whose fiber is in the homology class of $\alpha$. Blowing-up $X$ adds a $(-1)$-framed unknotted $2$-handle to the handle diagram, which can be slid over the $0$-framed $2$-handle of the round $2$-handle so that it is now attached to the fiber as a Lefschetz handle. (See Figure~\ref{trivialgenusoneSBLFs} for an example.) If we have a broken Lefschetz fibration which is \textit{not} relatively minimal, then the Lefschetz $2$-handles contributing the exceptional spheres on the fibers would have null-homotopic vanishing cycles. These are \linebreak $(-1)$-framed unknotted $2$-handles which do not link with the $1$-handles of the standard handlebody diagram, so we can apply the same trick above so as to get an SBLF of the same genus. 
\end{proof}

\begin{lem} \label{connected}
If $(X, \alpha)$ admits a broken Lefschetz fibration with connected round locus and embedded round image, then it also admits a homotopic SBLF (which possibly has fibers of different genera).
\end{lem}

\begin{proof}
The homotopy in question to obtain a broken Lefschetz fibration with connected fibers is presented in \cite{B2}, as a result of the flip-and-slip move introduced there. (See Figure~\ref{Figure: Flip and slip}.) If the round locus is empty, one can locally introduce a ``button'' described in \cite{P, B1} (equivalently, introduce a ``birth'' singularity and turn the two cusps into Lefschtez singularities as in \cite{L, W}). We can then apply the previous lemma to guarantee the relative minimality of the fibration. 
\end{proof}

\begin{figure}[ht] 
\begin{center}
\includegraphics[scale=0.6]{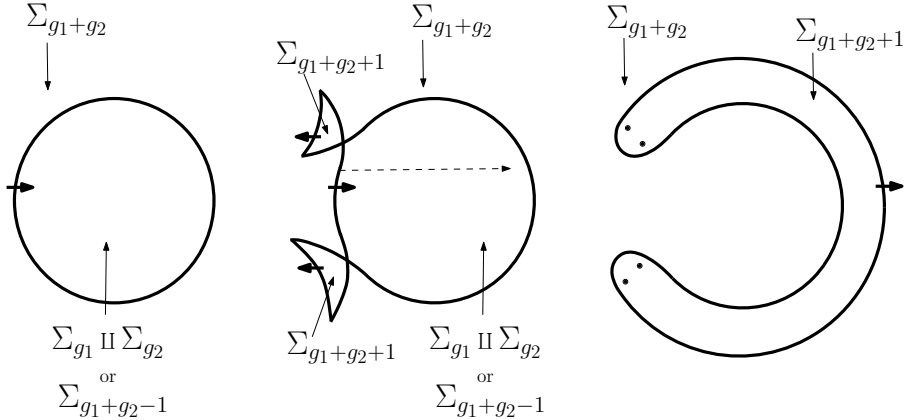}
\caption{\small ``Flip and slip'' move. On the left: The disk neighborhood of the round image of the original fibration with either disconnected fibers or with fibers of smaller genera in the interior. After two flips we get the fibration in the middle, and then slide the indefinite fold arc as indicated by the dashed arrow. On the right: 
The resulting fibration after the slip, where the four cusps are turned into Lefschetz critical points. The arrows indicate the direction of the fiberwise $2$-handle attachments.}
\label{Figure: Flip and slip}
\end{center}
\end{figure}

Given a closed oriented $4$-manifold $X$, and a self-intersection zero homology class $\alpha \in H_2(X)$, we can now prove that the invariants $g(X, \alpha)$, $S_g(X)$, and $g(X)$ are well-defined. The existence of an SBLF for a given pair $(X, \alpha)$ appears in \cite{W}, the proof of which we will outline here for completeness: 

One can represent $\alpha$ by an embedded orientable surface $F$ which has a trivial tubular neighborhood. The work of Gay-Kirby \cite{GK1} shows that there exists an achiral broken Lefschetz fibration with only connected fibers, having $F$ as one of its regular fibers, and satisfying the following properties: The fibration is injective on its critical set, the image of the round singularity is \textit{bi-directed}, that is it is composed of circles which are all parallel to the equator (which is composed of only regular values this time) contained in the interior of a big annular tubular neighborhood $B$ of the equator, where $B$ does not contain the images of any Lefschetz critical points, and such that the genera of the fibers only decrease when we trace the preimages over an oriented arc in $B$ connecting the equator to either one of the boundary components while intersecting the image of each round singular circle only once.

We can then eliminate the negative Lefschetz singularities using the handlebody argument in \cite{B3} or the singularity theory argument in \cite{L}. By pushing the Lefschetz singularities to the higher side of each round image using Lemma \ref{push}, we obtain a bi-directed broken Lefschetz fibration whose all Lefschetz critical points lie in a neighborhood of an annular neighborhood of the equator of $S^2$. As shown in \cite{W}, one can then employ the flip and slip move of \cite{B2} and the inverse merge move of \cite{L} so as to reduce the number of round circles to one. Then Lemma \ref{relativeminimality} above hands us a simplified broken Lefschetz fibration. Although these moves may (and typically will) alter the topology of fibers, they keep the homology class of the fiber intact. Hence, there always exists a simplified broken Lefschetz fibration on $X$ whose fiber is in the homology class $[F] = \alpha$. 

We can now define the broken genera invariants: Let $X$ be a closed smooth oriented $4$-manifold and $\alpha$ be a homology class in $H_2(X; \Z)$ with $\alpha^2 = m \geq 0$. Let $\tilde{X} \cong X\#m\CPb$ be the blow-up of $X$ at $m$ points and let $\tilde{\alpha} \in H_2(\tilde{X}; \Z)$ denote the homology class ${\iota}_*(\alpha)- \sum_i e_i$. Here $\iota$ is the inclusion of $X$ minus the blown-up points into $\tilde{X}$, and $e_i$ are the homology classes of the exceptional spheres. We define the \emph{broken genus} of $(X, \alpha)$ as the smallest non-negative integer $g$ among the genera of all possible SBLFs on $\tilde{X}$ whose fibers represent the class $\tilde{\alpha}$, and denote it by $g(X, \alpha)$. The \emph{broken genus series} of $X$, denoted by $S_g(X)$, is defined as the following formal sum in the integral group ring $\Z[H_2(X)]$ 
\[ 
S_g(X)  \, = \sum_{\alpha \in H_2(X)\, , \, \alpha^2=0} \, g(X, \alpha) \, t_{\alpha} \, + \sum_{\beta \in H_2(X)\, , \, \beta^2 >0} \, g(X, \beta) \, t_{\beta} \, ,
\]
and splits naturally as the sum of two power series on the right, which we will denote by $S^0_g(X)$ and $S^+_g(X)$. We then define the \emph{broken genus} of $X$, denoted by $g(X)$, as the smallest coefficient that appears in $S^0_g(X)$. 

Clearly $g(X, \alpha)$, $S_g(X)$ and $g(X)$ are all well-defined. Let $\psi: X' \to X$ be a diffeomorphism. If $f\colon\,X \to S^2$ is an SBLF with a regular higher side fiber $F$, then $F'=\psi^{-1}(F)$ is a regular higher side fiber of the SBLF $f'=f \circ \psi: X' \to X$ with $[F']= \psi^{-1}_*([F])$. This concludes our claim that broken genera invariants are smooth invariants. 

\begin{rk} \rm
The name \textit{simplified broken Lefschetz fibrations} went through some evolution since it was first introduced in \cite{B1} by the author, which is worth clarifying here: Without the assumption on connectedness or relative minimality of fibers, they were first shown to exist on blow-ups of all $4$-manifolds with $b^+>0$, namely on near-symplectic $4$-manifolds, which were the focus of \cite{B1}. The connectedness of fibers were later shown in \cite{B2}, again for near-symplectic broken Lefschetz fibrations, which follows from applying Lemma \ref{connected} in this setup. Williams worked out an algorithm for employing flip and slip and inverse merge moves as we reviewed above to get the existence of SBLFs (without the assumption on relative minimality) on arbitrary closed oriented $4$-manifolds \cite{W}. More recently in \cite{BK, H1,H2}, relative minimality assumption was added when studying these fibrations. 
\end{rk}

\begin{rk} \rm
The trivial homology class does not always give rise to an SBLF with the smallest broken genus, that is, $g(X, 0)$ is not necessarily equal to $g(X)$. For instance, any $S^2$-bundle over $S^2$ has a homologically essential fiber. So for $X=S^2 \x S^2$ or $\CP \# \CPb$, $g(X, 0) > g(X)=0$. 

\end{rk}

\begin{rk} \rm
As mentioned in the Introduction, the present work does \textit{not} deal with the possible connection between broken Lefschetz fibrations and the classical minimal genus problem for a given second homology class in a $4$-manifold $X$. To the author of the current article it seems most convenient to stay in the near-symplectic setting when addressing the minimal genus problem, and therefore pursuing this connection through a study of $S^+_g(X)$ above. We plan to take up this problem elsewhere. 
\end{rk}

\vspace{0.2in}
\section{Bounds for broken genera invariants} \label{Properties}

There are several bounds one can get on the broken genus based \emph{solely} on the topological type. We will note a few of them here.

\begin{lem}\label{euler}
If $X$ admits a genus $g$ SBLF with $k$ Lefschetz singularities, then \linebreak $\eu(X)=4-4g+k$ if round locus is empty, and $\eu(X)=6-4g+k$, otherwise. 
\end{lem}

\begin{proof}
If the round locus is empty, then $X$ decomposes into $D^2 \x \Sigma_g$ and a genus $g$ Lefschetz fibration over $D^2$. The latter admits a handlebody description with $k$ $2$-handles attached to $D^2 \x \Sigma_g$, so $\eu(X)= 2 (2-2g) +k = 4 -4g+k$.

If the round locus is non-empty, then $X$ decomposes into three pieces instead, $D^2 \x \Sigma_{g-1}$, a round  cobordism, and a genus $g$ Lefschetz fibration over $D^2$. Since the round cobordism has euler characteristic zero, we get 
\[ \eu(X)= 2-2(g-1) + 2-2g +k = 6-4g +k \, . \]
\end{proof}

\begin{lem}\label{pi1}
Let $f: X \to S^2$ be a genus $g$ SBLF with $k$ Lefschetz critical points. Then $2g \geq b_1(X) \geq 2g-k-2$ and $2g \geq b_1(X) \geq 2g - k -1$ if $f$ admits a section. If the round locus is empty, then the bounds on the right improve to $b_1(X) \geq 2g -k -1$ and $2g-k$, respectively.
\end{lem}

\begin{proof}
As shown in \cite{B1}, a genus $g$ SBLF gives rise to a handlebody decomposition of $X$ with $2g$ $1$-handles and $k+3$ \, $2$-handles, where the rest of the handles are of higher index. Here, $k$ of these $2$-handles are Lefschetz $2$-handles, one is for the fiber of the higher side, one comes from the round $2$-handle, and one last $2$-handle is pulled back from the lower side. The $2$-handle for the fiber goes over each $1$-handle algebraically zero times. If there is a section, then the last $2$-handle can be seen to be unlinked with all the $1$-handles. If we have an honest Lefschetz fibration instead, then we get a similar handlebody decomposition with $2g$ $1$-handles, and $k+2$ \, $2$-handles, without the $2$-handle of the round $2$-handle. Using these handlebody decompositions to calculate the first homology, we get the desired inequalities. 
\end{proof}

\begin{rk} \rm
Using similar arguments as in the proof of above lemma, it is easy to get constraints on the broken genus relying solely on $\pi_1(X)$, which however depends on the group of course. Such bounds would be handy when addressing the Problem~\ref{groupgenera} we rise at the end of our paper.
\end{rk}

We see that any non-negative integer value can be attained as the broken genus of some $4$-manifold:

\begin{prop} \label{unbounded}
$g(S^2 \x \Sigma_g)=g$ for any non-negative integer $g$. 
\end{prop}

\begin{proof}
The manifold $X=S^2 \x \Sigma_g$ admits the trivial SBLF of genus $g$, which is the genus $g$ surface bundle obtained by projection onto the first factor. If $X$ admits an SBLF of genus $h$, then by Lemma \ref{euler} we get $4-4g = \eu(X) \geq 4-4h +k$, so $h\geq g$, completing the proof. 
\end{proof}

This result can be generalized drastically, if we focus on getting a lower bound for the broken genus of a \emph{pair} $(X, \alpha)$. This follows from the symplectic Thom conjecture, proved by Morgan, Szab\'o and Taubes \cite{MST}, and thus provides a lower bound based on the \emph{smooth} topology:

\begin{prop} \label{Thom}
If $X$ admits a symplectic form for which there is an embedded symplectic surface $F$ in $X$ representing $\alpha \in H_2(X)$, then $g(X, \alpha) = g(F)$. In particular, if $(X, \alpha)$ admits a Lefschetz fibration on $X$ with a homologically essential fiber $F$ (in particular whenever $g(F) \neq 1$ or the critical locus is non-empty), then $g(X, \alpha)$ = $g(F)$. 
\end{prop} 

\begin{proof}
The first statement follows from the main theorem of Morgan, Szab\'o and Taubes in \cite{MST}, who proved that a symplectic surface $F$ in a symplectic $4$-manifold $X$ has the smallest genus among all surfaces that are in the same homology class. (The case $b^+(X)=1$ was shown earlier by Li and Liu \cite{LL}.) When we have a Lefschetz fibration with a homologically essential fiber, the Thurston-Gompf construction gives a symplectic form on $X$ making any fiber $F$ symplectic. 
\end{proof}

We now turn to getting some upper bounds on the broken genus of a $4$-manifold by looking at some topological operations which would allow us to get new SBLFs from old. Naturally, many of these modifications will be through indefinite generic maps. So we are going to express many of our intermediate results in terms of \linebreak SPWFs, which has the additional advantage of translating many of our later results to the context of SPWFs with ease. 

\begin{figure}[ht] 
\begin{center}
\includegraphics[scale=0.6]{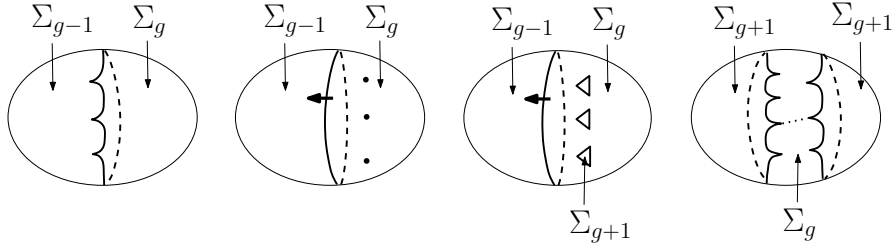}
\caption{\small The first two figures: The base diagram of a genus $g$ SPWF drawn with three cusps, and a genus $g$ SBLF obtained from it after ``unsinking'' the cusps and trading them with Lefschetz critical points. The second, third, and fourth figures: A genus $g$ SBLF drawn with three Lefschetz critical points turned into a genus $g+1$ SPWF. After perturbing all the Lefschetz critical points into indefinite singular circles with three cusps on each, we get the second figure from the right. After inverse merging all these singular circles that come from the Lefschetz critical points into one, and applying a flip and slip to the far left disk region without eliminating the cusps, we obtain the figure on the far right. We can then inverse merge the two indefinite circles with cusps to pass to a genus $g+1$ SPWF with five cusps.}
\label{Figure: SBLFvsSPWF}
\end{center}
\end{figure}

\begin{lem} \label{SBLFvsSPWF}
If $(X, \alpha)$ admits a genus $g$ SPWF, then it admits a genus $g$ SBLF. If it admits a genus $g$ SBLF, then it admits a genus $g+1$ SPWF. 
\end{lem}

\begin{proof}
The first claim follows from eliminating each cusp singularity by introducing a Lefschetz critical point on the higher genus side using the unsinking move \cite{L, W}. (See Figure~\ref{Figure: SBLFvsSPWF}.) Then the additional constraints met by the SPWF translates to those required for an SBLF. The second claim is proved as in the previous section based on an algorithm of Williams': We first perturb each Lefschetz singularity locally to an indefinite singular circle with three cusps. We can then inverse merge all these to get only one circle with many cusps, lying entirely on the higher side of the original fibration. Next we apply the flip and slip move without turning the resulting cusps into Lefschetz singularities. Thus we have an indefinite generic map on $X$ whose base diagram is as follows: On the far left and far right it has fibers of genera $g+1$, and in the middle it has fibers of genera $g$, where the two circles with cusps have cusps facing each other (Figure~\ref{Figure: SBLFvsSPWF}). Merging these two circles into one results in a genus $g+1$ SPWF. We finally note that $\alpha$, the homology of the fiber class, does not change during these homotopy moves.
\end{proof}

It immediately follows that:

\begin{prop} \label{orientation}
If $g(X, \alpha)= g$ for $\alpha^2=0$, then $g-1 \leq g(\bar{X}, \alpha) \leq g+1 $.
\end{prop}

\begin{proof}
Observe that the local model for cusp and fold singularities is \emph{not} oriented, so any genus $g$ SPWF on $X$ gives rise to a genus $g$ SPWF on $\bar{X}$ where the fibers are in the same homology class. From Lemma \ref{SBLFvsSPWF}, $g(X, \alpha)= g$ implies that $(X, \alpha)$ admits a genus $g+1$ SPWF, and thus so does $(\bar{X}, \alpha)$. In turn, we get an SBLF for $(\bar{X}, \alpha)$, implying that $g(\bar{X}, \alpha) \leq g+1$. Using the same trick, we see that $g(\bar{X}, \alpha) \leq g-2$ would lead to a contradiction to our assumption that $g(X, \alpha) = g$.  
\end{proof}

\begin{figure}[ht] 
\begin{center}
\includegraphics[scale=0.6]{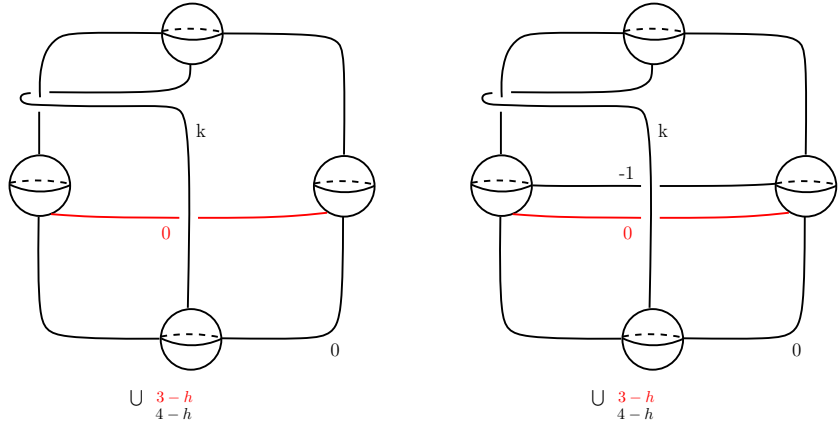}
\caption{\small On the left: The genus one SBLF on $S^4$. The twisted gluing of the round cobordism to the higher side along $T^3$ causes the $2$-handle from the lower side to be pulled back to the higher side so that it links with the $2$-handle of the higher genus fiber once. Its framing $k$, determined by the twisting, does not change the underlying topology. On the right: The genus one SBLF on $\CPb$, as an example of a blow-up of an SBLF with non-empty round locus.}
\label{trivialgenusoneSBLFs}
\end{center}
\end{figure}

\begin{rk} \label{ruled} \rm
The $\pm 1$ discrepancy between the broken genera of $(X, \alpha)$ and $(\bar{X}, \alpha)$ cannot be removed; there are many pairs $(X, \alpha)$ with $g(\bar{X}, \alpha)= g(X, \alpha)  + m$, for each $m = -1, 0, 1$. For instance, consider the genus one SBLFs on $S^4$ and $\CPb$ in Figure~\ref{trivialgenusoneSBLFs}. Since neither one of these manifolds is ruled, they can only admit SBLFs of positive genera and with fibers representing the trivial homology class. Clearly $(\bar{S^4},0)$ can be supported by the same SBLF, which realizes the smallest broken genus. On the other hand, from Lemma \ref{euler} we see that a genus one SBLF on $\CP$ would have a single Lefschetz singularity. It is easy to see that the only essential embedded curve that the monodromy of the higher side prescribed by a single Dehn twist along a non-separating curve $a$ on $T^2$ can fix (in order to hand us a round $2$-handle cobordism) is parallel to $a$. This gives an embedded sphere of self-intersection $-1$, which is impossible for $\CP$.
\end{rk}

We can now generalize the argument we gave in the proof of Lemma~\ref{relativeminimality} to obtain:

\begin{prop} \label{achiralstab}
Let $(X, \alpha)$ admit a genus $g$ SBLF with non-empty round locus (resp. with non-empty round locus and with at least one Lefschetz critical point) and let \, $M$ be $\# a \CP \# b \CPb $ for any $a, b \geq 0$ (resp. $\# r (S^2 \x S^2)$ for any $r \geq 0$). \linebreak Then $(X \# M, \tilde{\alpha})$ where the class $\tilde{\alpha}$ represents the image of $\alpha$ under the homomorphism induced by the inclusion map $X \setminus D^4 \hookrightarrow X\# M$, admits a genus $g+1$ SBLF. 
\end{prop}

\begin{proof}
Let $(X,f)$ be a genus $g$ SBLF on $(X, \alpha)$ with non-empty round locus. \linebreak If $M= \# a \CP \# b \CPb$, we can add $a$ many $(+1)$-framed and $b$ many  $(-1)$-framed unknotted $2$-handles to the standard handle diagram of $(X,f)$ and slide all of them over the $0$-framed $2$-handle of the round $2$-handle as in the proof of Lemma~\ref{relativeminimality}, so that their vanishing cycles lie on the fiber with $(+1)$- and $(-1)$-framings, respectively. This describes an \emph{achiral} SBLF on $X \# M$, with $a$ negative and $b$ positive new Lefschetz singularities introduced to the original SBLF $(X,f)$. We can then follow the proof of Lemma~\ref{SBLFvsSPWF} first to turn this genus $g$ achiral SBLF into a genus $g+1$ SPWF, and then the latter into a genus $g+1$ SBLF, supporting $(X\# M, \tilde{\alpha})$ as in the statement. 

When $M= \# r (S^2 \x S^2)$, this time we add $r$ pairs of $0$-framed unknots to the standard diagram of $(X,f)$, where each pair links once. We can then slide these pairs of $2$-handles over a Lefschetz $2$-handle so that each becomes a pair of \textit{unlinked} positive and negative Lefschetz handles. (This modification is nothing but the one observed in \cite{GS}, Example 8.4.6, now performed in the BLF setting.) Once again, this yields a genus $g+1$ SBLF supporting $(X\# M, \tilde{\alpha})$. 
\end{proof}

We are going to conclude this section by showing that there is an additive upper bound on connected sums, which will serve as a strong calculational tool.

\begin{thm} \label{SPWFconnected}
Let $f_i\colon\,  X_i \to S^2$ be a genus $g_i$ SPWF with a section $S_i$, for $i=1,2$. Then there exists an SPWF (resp. SBLF) $f\colon\, X=X_1 \# X_2 \to S^2$ of genus $g_1 + g_2$ with a section $S$. If $\alpha_i$ is the homology class of the fiber of the SPWF (resp. SBLF) $f_i$, then the homology class of the fiber of $f$ is  $(\iota_1)_* \alpha_1 + (\iota_2)_* \alpha_2$, where $\iota_i$ is the inclusion $X_i \setminus D^4 \hookrightarrow X_1 \# X_2$. 
\end{thm}

\begin{proof}
By assumption, we can remove the tubular neighborhood of a higher genus (resp. lower genus) fiber of $f_1: X_1 \to S^2$ (resp. of $f_2: X_2 \to S^2$) so that the restriction of $f_1$ (resp. $f_2$) to the complement admits a disk section $D_1$ (resp. $D_2$) which restricts to a trivial $1$-braid on the boundary $S^1 \x \Sigma_{g_1}$ (resp. $S^1 \x \Sigma_{g_2-1}$). We can moreover isotope all the self-intersection points of these sections so that they are contained in these fiber neighborhoods. We can now perform a ``large'' version of the connected sum operation of \cite{P, B1} along these two disk sections: \linebreak We take out the disk fibered tubular neighborhoods of each $D_i$ and identify \linebreak $D_1 \x \partial D^2$ with $D_2 \x \partial D^2$ so as to extend the fibrations on the rest. The local model for the round $2$-handle cobordism from the trivial boundary fibrations on $S^1 \x \Sigma_{g_1} \# \Sigma_{g_2-1}$ to $S^1 \x (\Sigma_{g_1} \sqcup \Sigma_{g_2-1})$ is the same as the one for the usual connected sum operation, along which we identify $\partial D_1 \x D^2$ with $\partial D_2 \x D^2$ so that the underlying topological modification amounts to gluing $X_1 \setminus \, D_1 \x D^2$ and $X_2 \setminus \, D_2 \x D^2$ along their boundaries, resulting in $X=X_1 \# X_2$.

\begin{figure}[ht] 
\begin{center}
\includegraphics[scale=0.6]{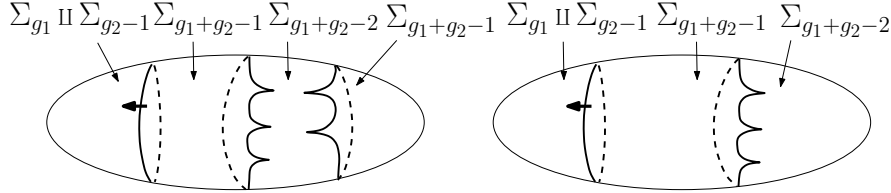}
\caption{\small On the left: The base diagram for the fibration we get after the ``large'' connected sum along the disk sections. The singular circle coming from $f_1$ is depicted with three cusps and that of $f_2$ with two. On the right: The base diagram after inverse merging the two indefinite singular circles with cusps.}
\label{Figure: ConnectedSum}
\end{center}
\end{figure}

We obtain a new fibration on $X=X_1\# X_2$ as depicted on the left in Figure~\ref{Figure: ConnectedSum}. The fibers over the four regions are as follows: On the far left, we have disconnected fibers of genera $g_1$ and $g_2-1$ over the $2$-disk, where the fibration restricts to trivial surface bundles on each piece. Over the neighboring annular region we have connected fibers of genera $g_1+g_2-1$. We have the simple round singularity introduced by the connected sum operation in between these two regions. Over the next neighboring annular region we have fibers of genera $g_1+g_2-2$. The fibered cobordism between these two annular regions comes from the singular circle of $f_1$ in $X_1$. Finally, on the far right, we have a disk region over which fibers have genera $g_1+g_2-1$. The fibered cobordism between this disk region and its neighboring annular region comes from the singular circle of $f_2$ in $X_2$. There are cusps on both sides of the third region pointing inwards.

We now proceed as follows: We can first combine the two singular circles with cusps on the right by inverse merging them along an arc between a cusp point on each which projects to an embedded arc on the base. At this point we have an indefinite generic map with a base diagram as given on the right in Figure~\ref{Figure: ConnectedSum}. The fibers over the three regions are as follows: The region on the far left is the same as the far left region of the previous fibration, over which we have disconnected fibers of genera $g_1$ and $g_2-1$. The region on the far right has fibers of genera $g_1+g_2 -2$. The annulus region has fibers of genera $g_1+g_2-1$, and all the cusp points are contained on the singular circle on its right, pointing into the disk region. 

We can now employ the flip and slip move to both ends of this fibration, \textit{without eliminating the cusp points at the end}. We now have three regions for this fibration, where the fibers on the far left and far right are both connected and of genera $g_1+g_2$, and the cusps are pointing into the annular region in the middle. Inverse merging these two singular circles with cusps, we obtain a genus $g_1+g_2$ SPWF on $X$. Clearly, we can turn all the cusps into Lefschetz critical points at the end to obtain a genus $g_1+g_2$ SBLF on $X$ as well.

A push-off of the section $S_i$ of $f_i\colon\,  X_i \to S^2$ for each $i = 1, 2$ will give rise to a section of $f\colon\, X \to S^2$ at the end, since all the homotopies involved above can be performed away from $S_i$. As for the last assertion on the homology of the new fiber, observe that in the very first step of our construction, the homology of the fiber over the far left region is represented by the disjoint union of the inclusion of the fibers of $f_1$ and $f_2$. The following homotopy moves keep this homology class intact.  
\end{proof}

One of the subtleties in modifying SBLFs as in Theorem \ref{SPWFconnected} is due to the fact that the singular circle of the fibration may not contain cusps nor could be turned into one without increasing the genus. Although an indefinite cusp point can be turned into a Lefschetz critical point, this trade does not always work backwards. One can \emph{sink} a Lefschetz critical point into an indefinite fold singular circle and introduce a new cusp point, if for a nearby reference fiber the vanishing cycle of the Lefschetz singularity and the vanishing cycle of the indefinite fold circle intersect transversally at one point. In this case, we are going to call this Lefschetz critical point \emph{sinkable} ---to an indicated (part of an) indefinite fold circle when it is not an SBLF. The next corollary now follows from the proof of the above theorem in a straightforward way, as we keep pushing the Lefschetz critical points to the higher sides:

\begin{cor} \label{SBLFconnected}
Let $f_i\colon\,  X_i \to S^2$ be a genus $g_i$ SBLF with a section $S_i$, for $i=1,2$. Let $f_1: X_1 \to S^2$ have non-empty round locus and a sinkable Lefschetz critical point. Then there exists an SBLF $f\colon\, X=X_1 \# X_2 \to S^2$ of genus $g_1+ g_2$, if $(X_2, f_2)$ has empty round locus or a sinkable Lefschetz critical point, and of genus $g=g_1+g_2+1$ otherwise. If $\alpha_i$ is the homology class of the fiber of the SBLF $f_i$, then the homology class of the fiber of $f$ is  $(\iota_1)_* \alpha_1 + (\iota_2)_* \alpha_2$, where $\iota_i$ is the inclusion $X_i \setminus D^4 \hookrightarrow X_1 \# X_2$. 
\end{cor}

\noindent When the genus $g_1$ SBLF $f_1: X_1 \to S^2$ has no round locus or no sinkable Lefschetz critical point, after a single birth or flip and slip move, we can switch to a genus $g_1+1$ SBLF and apply the corollary.

\vspace{0.2in}
\section{Broken genera and exotic simply-connected $4$-manifolds} \label{Exotic}

Freedman's seminal work in \cite{Freedman}, together with Donaldson's diagonalization result \cite{Donaldson}, allows us to pin down the homeomorphism type of a simply-connected $4$-manifold by looking at its intersection form. In particular a non-spin simply-connected $4$-manifold is homeomorphic to $a \CP \# b \CPb$ for some $a,b$, and if spin, assuming that the $11/8$ conjecture holds, it is homeomorphic to $m \K \# n (S^2 \x S^2)$ with one of the orientations. Our first theorem in this section calculates the broken genera of these standard simply-connected $4$-manifolds: 

\begin{thm} \label{constructions}
The only $4$-manifolds with broken genus zero are $S^2 \x S^2$ and $\CP \# \CPb$. The $4$-manifolds $S^4$, $\# r (S^2 \x S^2)$ for $r \geq 2$, $a \CP \# b \CPb$ for \linebreak $a \geq 0, b \geq 1$ except for $a=b=1$, and $\K$, all have broken genus one. The $4$-manifolds $\# m \CP$, for $m \geq 1$, and $\Kb$ have broken genus two.
\end{thm}

\begin{proof}
The first statement is obvious; see Remark \ref{ruled} above. So any manifold admitting a genus one SBLF has broken genus one, unless it is diffeomorphic to $S^2 \x S^2$ or $\CP \# \CPb$. We immediately see that the broken genera of $S^4$, $\CPb$ (see Figure~\ref{trivialgenusoneSBLFs}), and $\K$ are one.

We are going to prove that $g(\# r (S^2 \x S^2))=1$ by induction. We claim that for each $r \geq 2$, there exists a genus one SBLF on  $\# r (S^2 \x S^2)$ with non-empty round locus, with sinkable Lefschetz critical points, and with a section. Since $S^2 \x S^2$ admits a genus zero bundle with a section, by Corollary \ref{SBLFconnected}, $\# 2 (S^2 \x S^2)$ admits a genus one SBLF satisfying all the conditions listed above. We can then induct on $r \geq 2$ by invoking Corollary \ref{SBLFconnected} again. 

For $a=b > 1$, one can prove the analogous result for $a \CP \# b \CPb$ in the same way as above. However to realize all possible pairs of $a,b$, we will take a different approach. We are going to first show by induction that for each $a \geq 2$, $\# a \CP \# \CPb$ admits a genus one SBLF with non-empty round locus, with sinkable Lefschetz critical points, and with a section. After that, using Lemma \ref{relativeminimality}, we can get genus one SBLFs on any blow-up of $a \CP \# \CPb$.

\begin{figure}[ht]  
\begin{center}
\includegraphics[scale=0.7]{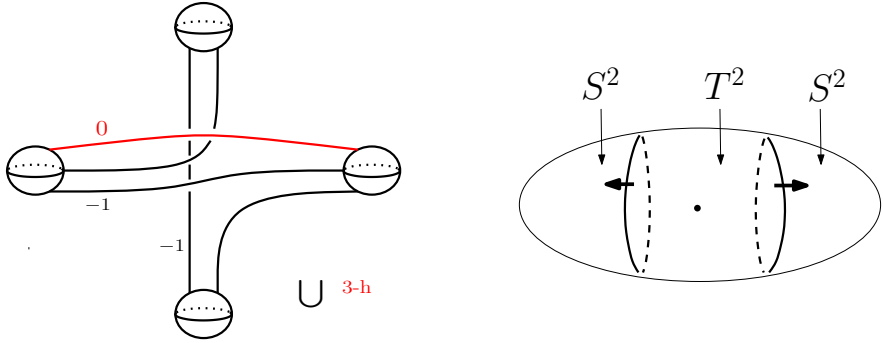}
\caption{\small On the left: The local model after birth. The vanishing cycles of the two Lefschetz singularities and the round $2$-handle lie on the once punctured torus as depicted in the handle diagram. The round $2$-handle, as a union of a $2$-handle and a $3$-handle, is given in red. On the right: The broken Lefschetz fibration on $2\CP \# \CPb$, before simplifying it to get an SBLF.}
\label{Figure: Birth and ADK}
\end{center}
\end{figure}

Start with any genus zero bundle on $\CP \# \CPb$ with a section. Introduce a birth to get a genus one SBLF on $\CP \# \CPb$ with two Lefschetz critical points, which are clearly sinkable. (See Figure~\ref{Figure: Birth and ADK}.) We can locally replace one of these Lefschetz critical points with a round singular circle after the operation given in Example~3 of \cite{ADK}. For a chosen small disk neighborhood of the Lefschetz critical value, the new round singularity is introduced so that its image is an embedded curve parallel to the boundary of this disk and the fibers in the interior are of smaller genera. Importantly, for a reference fiber on the boundary of this disk, the vanishing cycle of this new round $2$-handle is precisely the vanishing cycle of the old Lefschetz critical point. As shown in \cite{ADK}, this amounts to an anti-blow-up at the Lefschetz critical point, that is the new $4$-manifold we get is a connected sum of the old one with $\CP$. We can certainly perform this operation away from the section so that it embeds as a section of the new fibration. 

At this point we have an SBLF with three regions: In the middle we have a genus one Lefschetz fibration with one Lefschetz critical point, and on the sides we have two $S^2$ bundles which are round $2$-handle cobordant to the fibration in the middle. (See Figure~\ref{Figure: Birth and ADK}.) Tracing the initial configuration of curves given by the birth throughout these steps, we can see that the vanishing cycles of the round $2$-handles intersect at one point on a reference fiber in the middle. Therefore, one can perform the inverse merge move to mash these two round circles into one \cite{L, W}. Hence we get a genus one SBLF on $2 \CP \# \CPb$ with a section and with a sinkable Lefschetz singularity. By applying the same procedure as above, this time without the birth and starting with the anti-blow-up modification along a sinkable Lefschetz critical point, we can induct on $a$ to get the desired genus one SBLF on $a \CP \# \CPb$, for any $a \geq 2$. This completes the genus one case. 

Analyzing the possible monodromies, Hayano shows in \cite{H1} that $ \# m \CPb$ does not admit genus one SBLFs. Since $\# m \CPb$ admits a genus one SBLF for each $m \in \Z^+$, $\# m \CP$ admits a genus two SBLF by Proposition ~\ref{orientation}, realizing its broken genus. On the other hand, another result of Hayano's is that if a simply-connected spin $4$-manifold admits a genus one SBLFs with non-empty round locus, then it has zero signature \cite{H2}. Clearly $\Kb$ does not admit a genus one Lefschetz fibration, nor does it admit a genus one broken Lefschetz fibration. However, it admits a genus two SBLF, again by Proposition \ref{orientation}. 
\end{proof}

\begin{rk} \rm
Extending the well-known classification of genus one Lefschetz fibrations, Seiichi Kamada and the author obtained a classification of genus one SBLFs in \cite{BK}. (Also see Hayano's work \cite{H1} for SBLFs with less than six Lefschetz critical points.) Low genera SBLFs on some of the above $4$-manifolds were given in \cite{BK}, \cite{H1} and \cite{HS1}, using explicit monodromy and handlebody descriptions followed by rather lengthy Kirby calculus. The first two articles show that there are other genus one SBLFs on standard non-simply-connected $4$-manifolds.
\end{rk}

\begin{cor} \label{SW}
Let $X$ be a simply-connected $4$-manifold with $b^+(X)>1$, not diffeomorphic to an elliptic surface. If $X$ has a non-trivial Seiberg-Witten invariant, then $g(X) >1$. 
\end{cor}

\begin{proof}
As shown in \cite{BK}, any \textit{simply-connected} $4$-manifold $X$ admitting a genus one SBLF is either an elliptic surface or decomposes as a connected sum of $\CP$s with $\CPb$s after a certain number of blow-ups. In the latter case, our assumption $b^+(X) > 1$ implies that we get more than one copy of $\CP$ in such a decomposition, which would yield the vanishing of the Seiberg-Witten invariant. However, if $X$ has a non-trivial Seiberg-Witten invariant, then so does any blow-up of it. So $g(X) > 1$.
\end{proof}

\begin{rk} \rm
The reader will notice that the proof of the above corollary has a certain formalism, which allows us to rephrase the corollary, say by using $4$-dimensional Ozsv\'ath-Szab\'o invariants or Bauer-Furuta's stable cohomotopy Seiberg-Witten invariants instead.
\end{rk}

Coupling the classification of genus one SBLFs in \cite{BK} and Corollary \ref{SW}, we get:

\begin{cor} \label{symplectic}
Any symplectic $4$-manifold which is neither elliptic nor homeomorphic to a rational surface has broken genus greater than one.
\end{cor}

\begin{proof}
The only symplectic $4$-manifolds that appear in the classification of genus one SBLFs are the elliptic surfaces $S^2 \x T^2$, $E(n)$, and possibly some exotic $m\CP \# n \CPb$s. (See \cite{BK}, Corollary 14.) Since any symplectic $4$-manifold $X$ with $b^+(X)>1$ has non-trivial Seiberg-Witten invariant, Corollary \ref{SW} rules out the possibility $m > 1$. 
\end{proof}

In Proposition \ref{unbounded} we showed that the broken genus invariant could take arbitrarily high integer values by calculating the invariant on a family of non-homeomorphic $4$-manifolds. Our next theorem is an analogue of it for pairs, which is however sharper, as we show that there is a family of pairs $(X_n, \alpha_n)$ \emph{all in the same homeomorphism class} (i.e. the homeomorphisms match the homology classes as well), whose broken genera get arbitrarily high.

\begin{thm} \label{knotsurgery}
There is an infinite family of pairs $(X_n, \alpha_n)$, $ n \in \Z^+$, all homeomorphic to $(\K, \alpha)$, where $\alpha$ is the homology class of an elliptic fiber, such that no two of these $4$-manifolds are diffeomorphic and $g(X_n, \alpha_n)$ can take arbitrarily high integer values. Moreover, we can choose $X_n$ so that they are all symplectic or none admits a symplectic structure.  
\end{thm}

\begin{proof}
Let $K$ be a genus $g$ knot in $S^3$, where $g$ is the smallest genus of a Seifert surface for $K$. If $M_K$ is the $3$-manifold obtained by a $0$-surgery on $K$ with respect to this minimal genus Seifert surface, then it inherits a circle valued Morse function with only index one and two critical points such that the smallest genus among all the regular fibers is $g$. If $K$ is a fibered knot, then we can assume that the Morse function has no critical points. We extend this Morse function to a broken Lefschetz fibration with no Lefschetz singularities from $S^1 \times M_K$ to $T^2$ by identity on the first $S^1$ component. Here $S^1$ times the meridian of $K$ (viewed in $M_K$) gives a self-intersection zero torus $S_K$ which is a section of the broken Lefschetz fibration avoiding the singular set. We can compose this broken Lefschetz fibration with a double branched covering from $T^2$ to $S^2$ so that the four branched points of the latter are disjoint from the round image of the former on $T^2$. So we get a mixed fibration with round singularities and four multiple fibers of multiplicity two ---where the two disjoint regular fiber components meet. Moreover, $S_K$ becomes a $2$-section of this fibration, intersecting each fiber \textit{component} at one point.

Let $X = \K$ and $F$ be a regular elliptic fiber. One can construct $\K$ as the double branched cover along a singular set which is the union of four disjoint copies of $S^2 \x \{pt\}$ and four disjoint copies of $\{pt\} \x S^2$, desingularized at the 16 double points in a standard way. Precomposing the branched covering map with the projection on to each factor of $S^2 \x S^2$ gives a holomorphic ``horizontal'' and a ``vertical'' elliptic fibration, each with four multiple fibers with multiplicity two. Regarding the vertical fibration as the one that perturbs to the standard elliptic fibration on $X$, we see that any regular torus fiber of it is a self-intersection zero $2$-section of the horizontal fibration. 

Lastly, let $X_K$ be the $4$-manifold obtained from $X$ by a knot surgery along $F$ and using the knot $K$. This construction can be viewed as a generalized fiber sum of $X$ and $S^1 \x M_K$ along $F$ and $S_K$. We can view this as a sum along the $2$-sections of the two fibrations, matching the two fibrations on the complements while aligning the four multiple fibers of the same multiplicity. This results in a mixed fibration which is locally holomorphic around the multiple fibers and have only round singularities in the complement. We can perturb this fibration so that the multiple fibers give rise to Lefschetz singularities as argued in \cite{FS: same SW}, handing us a broken Lefschetz fibration with fibers of genera $2g+1$, and higher if $K$ is non-fibered. When $K$ is fibered, we obtain a genuine genus $2g+1$ Lefschetz fibration. Otherwise, it is \emph{not} simplified, but can be homotoped to one, which is a sure deal by Williams' main theorem \cite{W}. 

A crucial observation is that any regular fiber $T$ of the standard elliptic fibration on $X$ which is disjoint from $F$ we used in our construction above intersects the fiber $F$ of the resulting broken Lefschetz fibration on $X_K$ at two points positively.

We are now ready to describe the families $(X_n, \alpha_n)$ we will use. Let $K_n$ be a knot with the degree of its symmetrized Alexander polynomial equal to $n$, and set $X_n= X_{K_n}$. Recall that the knot genus is bounded below by the degree of the symmetrized Alexander polynomial, and this bound is sharp when the knot is fibered. Let $f_n: X_n \to S^2$ be the simplified broken Lefschetz fibration with a regular fiber $F_n$ one gets by the construction we described above. This is an honest genus $2n+1$ Lefschetz fibration when $K$ is fibered. Let $\alpha_n \in H_2(X_n)$ denote the homology class of the fiber of this fibration. Note that $n=0$ yields $X= X_0$ and the standard elliptic fibration $f=f_0$, whose regular fiber $F$ is in the class $\alpha = \alpha_0$. We can choose the generators of $H_2(X)$ so that two of them are the classes represented by a regular fiber $F$ and a $(-2)$-section of the standard elliptic fibration on it. Thus there is an isomorphism between the intersection forms of $X_K$ and $X$ matching the generator $\alpha_n$ with $\alpha$, which can be realized by a smooth h-cobordism, as shown by C.T.C Wall. It then follows from Freedman's work \cite{Freedman} that this h-cobordism can be \emph{topologically} turned into a trivial product cobordism, giving us the claimed homeomorphisms from $(X_n, \alpha_n)$ to $(X, \alpha)$ for any $n \geq 1$. 

All Seiberg-Witten basic classes of $X_K$ are multiples of one cohomology class, namely the Poincar\'e dual of $[T] \in H_2(X_K)$ ---viewed as the image of $[T]$ in $H_2(X \setminus \nu F)$ under the obvious inclusion. The Seiberg-Witten invariant of $X_K$ is equal to $\Delta_K(t)$, where $t= exp(2 PD([T]))$ and $\Delta_K$ is the symmetrized Alexander polynomial of $K$ \cite{FS: knot surgery}. Since each $K_n$ has a symmetrized Alexander polynomial of degree $n$, we immediately see that $X_n \cong X_m$ if and only if $n=m$, since otherwise the Seiberg-Witten invariants of $X_n$ and $X_m$ differ.

We can choose all $K_n$ fibered, thus for each $n$, we have a genus $2n+1$ Lefschetz fibration $f_n: X_n \to S^2$, which can be supported by a symplectic structure. By Proposition~\ref{Thom}, $g(X_n, \alpha_n) = 2n +1$, which constitutes a strictly increasing sequence for $n \in \Z^+$. So $(X_n, \alpha_n)$ make up the symplectic family in the statement of the theorem, where for each $n$, we constructed the SBLF realizing the broken genus for $(X_n, \alpha_n)$ above. 

If we instead choose all $K_n$ non-fibered and with non-monic degree $n$ symmetrized Alexander polynomials, then we get the second promised family of $(X_n, \alpha_n)$, which do not admit symplectic structures. To see this, suppose that $(X_n, \alpha_n)$ admits some genus $h$ \, SBLF, with a regular fiber $H$. So $[H]= \alpha_n = [F_n]$. From the knot surgery formula for the Seiberg-Witten invariants above we see that $2n PD([T])$ is a Seiberg-Witten basic class for $X_n$. By the Seiberg-Witten adjunction inequality, 
\[ -\chi(H) \geq [H]^2 + |2n \, [T] \cdot [H]| \, , \]
which implies $2h -2 \geq 4 n $. That is $h \geq 2n +1$, so $g(X_n \alpha_n) \geq 2n+1$ gets arbitrarily large for $n$ sufficiently large.
\end{proof}

\vspace{0.2in}
\section{Broken genus versus minimal genus} \label{BrokenVersusMinimal}

Given a closed oriented $4$-manifold $X$ and $\alpha \in H_2(X; \Z)$, let $mg(\alpha)$ denote the \emph{minimal genus} of $\alpha$, the smallest genus among the genera of all possible closed orientable surfaces in $X$ representing $\alpha$. For $\alpha^2 \geq 0$, for which both invariants are defined, it is clear that $mg(\alpha)$ gives a lower bound for $g(\alpha)$. As we noted in Proposition~\ref{Thom}, if $X$ admits a Lefschetz fibration of genus $g \neq 1$ with a regular fiber representing the homology class $\alpha$, then $g(\alpha)=mg(\alpha)=g$. Thus this inequality is sharp.

The minimal genus is a smooth invariant which is insensitive to the orientation of the ambient $4$-manifold, that is, $mg(\alpha)$ is unchanged under orientation reversal of $X$. Since we only allow Lefschetz singularities compatible with the orientation on the ambient $4$-manifold, it should be expected that $g(\alpha)$, for a class $\alpha^2=0$, would behave differently under the orientation reversal. Indeed, Proposition~\ref{orientation} along with Remark~\ref{ruled} demonstrate that this discrepancy occurs but in a rather controlled way: The difference between $g(\alpha)$ in $X$ and $\bar{X}$ can be $-1,0,$ or $1$, where all these possible values can be realized.

We can compare the broken genus and minimal genus for pairs $(X, \alpha)$ and $(X', \alpha')$ related to each other in more general ways than just the orientation reversal on the same smooth manifold. In particular we can compare the difference $g(\alpha)-mg(\alpha)$ to $g(\alpha')-mg(\alpha')$, so as to see if the minimal genus coupled with a certain equivalence between $(X, \alpha)$ and $(X', \alpha')$ would determine the broken genus. 

\noindent (i) Given $X$ and $X'$ closed oriented $4$-manifolds for which we have an isomorphism $\phi: H_2(X)\cong H_2(X')$, along with $\alpha \in H_2(X)$ and $\alpha'=\phi_*(\alpha) \in H_2(X')$, one can compare the difference between the broken genus and the minimal genus for pairs $(X,\alpha)$ and $(X',\alpha')$. We claim that such an equivalence does not provide any bound on the discrepancy between the broken and minimal genera of the matching classes. For example, set $X=S^2 \x S^2$ and $X'=S^2 \x \Sigma_g$ and let $\phi$ be an isomorphism between their second homologies matching the class $\alpha= [S^2 \x \{pt \}]$ in $X$ with the class $\alpha'=[S^2 \x \{pt \}]$ in $X'$. Proposition~\ref{unbounded} shows that the smallest broken genus of \emph{any} class in $S^2 \x \Sigma_g$ is $g$. So we have $mg(\alpha)=mg(\alpha')=0$, and $g(\alpha)=0$, whereas $g(\alpha') \geq g$ can be made arbitrarily large, by choosing $g$ large. 

\noindent (ii) We can restrict the above setting further and look at $4$-manifolds $X$ and $X'$ for which there is an isomorphism between their cohomology rings, inducing an isomorphism $\phi: H_2(X)\cong H_2(X')$ matching $\alpha$ with $\alpha'=\phi_*(\alpha)$. It is still not hard to find examples where a discrepancy ---greater than one--- occurs. For instance, let $X$ be $S^4$ or $\# m \CPb$, and $X'$ be a non-diffeomorphic $4$-manifold whose cohomology ring is isomorphic to that of $X$. There are infinitely many homology $4$-spheres with non-trivial fundamental groups, which would yield to such examples. It is clear that here the only candidate for a fiber class of an SBLF is the trivial class $\alpha=0=\alpha'$, which has minimal genus zero in any $4$-manifold. We know from Theorem~\ref{constructions} that $g(X,0)=1$, and moreover that $g(X',0)> 1$, since the classification result of \cite{BK} tells that if $X'$ were to admit a genus one SBLF, then after blow-ups it would decompose as a connected sum of $\CP$s and $\CPb$s with an indefinite intersection form. Hence, $g(\alpha)-mg(\alpha)=1$ whereas $g(\alpha')-mg(\alpha') \geq 2$.

\noindent (iii) Lastly, we can look at homeomorphic pairs of $4$-manifolds $X$ and $X'$ and classes $\alpha$ and $\alpha'$ matched under this homeomorphism. In Theorem~\ref{knotsurgery}, we studied families of non-symplectic $(X_n,\alpha_n)$ which are all homeomorphic to $(\K, [T])$, where $T$ is an elliptic fiber in $\K$. In this construction, one can use a carefully chosen circle valued Morse function without extrema on the complement of the non-fibered knot $K$ of genus $g$ so as to produce a broken fibration (without Lefschetz critical points) $S^1 \x M_K \to T^2$ with smallest fiber genus $g$, which in turn hands us a \emph{non-simplified} broken Lefschetz fibration on $X_K$ which has fibers of genus $2g+1$. The Seiberg-Witten calculation in the proof of Theorem~\ref{knotsurgery} shows that this is indeed the minimal genus of the resulting homology class of the fiber. However, since the $4$-manifold $X_K$ cannot be symplectic, we cannot have a Lefschetz fibration (of genus $2g+1$ or else) on it, implying that the broken genus of this class would be larger than its minimal genus. Setting $(X,\alpha)=(\K, [T])$, and $(X', \alpha')=(X_n, \alpha_n)$ for any one of these examples, we see that the difference $g(\alpha)-mg(\alpha)=0$, whereas $g(\alpha')-mg(\alpha') > 0$.

It might be conceptually helpful to note that given a closed oriented $4$-manifold $X$ and a surface $F$ which is a minimal genus representative for a class $\alpha \in H_2(X)$ with $\alpha^2=0$, one can always obtain a non-simplified broken Lefschetz fibration for $(X,\alpha)$ so that $F$ is a fiber of the resulting fibration. See for instance \cite{B2}, or Section~\ref{Preliminaries} where we sketched Gay-Kirby construction \cite{GK1} followed by an elimination of achiral points. Such a construction provides no a priori upper bound on the highest genus fiber (even when all the fibers are connected so as to make sense of this). Nevertheless, the lowest possible genus of a fiber of a general broken Lefschetz fibration on $(X, \alpha)$ is a smooth invariant, \textit{which is the same as the minimal genus} of $\alpha$ as the above argument demonstrates. With this in mind, we can compare the lowest possible genus of a fiber of a general broken Lefschetz fibration to that of a \emph{simplified} one (which differs from our broken genus invariant by at most one) on pairs $(X,\alpha)$ and $(X',\alpha')$.

The examples we produced under various equivalences (i)-(iii) above show that the difference between the lowest possible genera of simplified broken Lefschetz fibrations and of arbitrary broken Lefschetz fibrations on $(X,\alpha)$ and the same difference on $(X',\alpha')$ can still be: arbitrarily large under the equivalence in (i), and is at least one under the equivalence in (ii). For the families given in (iii), our current arguments do not suffice to show that the difference is still non-zero. However, we predict that for a family of non-fibered knots for which a circle valued Morse function as above requires at least $k\geq 1$ index one critical points, the number $k$ would provide a lower bound for the discrepancy between lowest genus of a simplified broken Lefschetz fibration and that of a general broken Lefschetz fibration, whereas this difference is zero for $(\K,[T])$. In any account, it is interesting to note that \emph{\emph{\textit{if this difference were to be constant for all homeomorphic pairs}}} $(X,\alpha)$ and $(X', \alpha')$, then the arguments in (ii) above would imply that there are \emph{no} exotic $S^4$s, or $\# m \CPb$s.

\vspace{0.2in}
\section{Related questions and problems}\label{Questions}

We finish with a few related problems. 

\begin{problem} \rm
Is there any closed oriented simply-connected $4$-manifold $X$ admitting a genus one simplified broken Lefschetz fibration and is not diffeomorphic to any one of the $4$-manifolds given in Theorem \ref{constructions}?
\end{problem}

\noindent Remarkably, from the signature zero condition Hayano obtained for spin genus one SBLFs in \cite{H2} it follows that the answer to this question is ``No'' if the $4$-manifold is spin. In the light of the classification of genus one SBLFs in \cite{BK} and Corollary \ref{SW} above, the question comes down to the existence of an exotic $\#m \CP \# n \CPb$ which, after a number of blow-ups, completely decomposes as a connected sum of $\CP$s and $\CPb$s, and has trivial Seiberg-Witten invariants if $m> 1$.

Thinking of the simply-connected $4$-manifolds that appear in Theorem \ref{constructions} as the summands of standard $4$-manifolds realizing all homeomorphism classes of simply-connected $4$-manifolds (modulo 11/8 conjecture), it remains to determine the broken genera of the following family of $4$-manifolds. Recall that the connected sum of $\K$ and $\Kb$ dissolves as $\# 22 (S^2 \x S^2)$.  

\begin{problem} \rm
Find the broken genera of $\# m \K \# n (S^2 \x S^2)$, for $m +n \geq 2$, with each orientation. 
\end{problem}
\noindent Note that by Proposition~\ref{achiralstab}, this problem essentially boils down to determining the genera of $\# m \K$, for $m \geq 2$. 

Any smooth $4$-manifold $X'$ homeomorphic to the $\K$ surface has $g(X') \geq g(\K)$. Moreover, the classification of genus one Lefschetz fibrations, together with Hayano's signature zero condition for spin genus one SBLFs with non-empty round locus, shows that the equality holds only if $X$ is diffeomorphic to $\K$. If $T$ is a regular elliptic fiber of the standard fibration on $\K$, then $g(\K)=g(\K, [T])=1$. So for any pair $(X', \alpha')$ homeomorphic to $g(\K, [T])$, we have $g(X', \alpha') \geq g(\K, [T])$, and the equality holds only if $(X', \alpha')$ is diffeomorphic to $(\K, [T])$. Clearly, the same holds for the homology class of the fiber of the ``dual'' elliptic fibration on $\K$. (Smoothly, these ``dual'' fibrations are the horizontal and vertical fibrations discussed in the proof of Theorem \ref{knotsurgery}.) 

\begin{problem} \rm
Let $X'$ be a smooth (resp. symplectic) $4$-manifold and $\psi: \K \to X'$ be a homeomorphism. Is it true that for any $\alpha'$ (resp. for any $\alpha'$ represented by an embedded symplectic surface) in $H_2(X')$, we have $g(X', \alpha') \geq g(\K, \psi^{-1}_*(\alpha'))$?
\end{problem}

\noindent A close look at the families of exotic $\K$ surfaces constructed in the proof of \linebreak Theorem~\ref{knotsurgery} shows that while for some homology classes the difference between broken genera can be arbitrarily large, it can remain the same for the others. 

Let $G$ be a finitely presented group. We can define a non-negative integer valued invariant of this group $G$ via broken genera: We set \emph{the broken genus of $G$}, denoted by $g(G)$, as the smallest broken genus among the broken genera of all possible closed oriented $4$-manifolds with fundamental group $G$. This gives ---yet another--- invariant of finitely presented groups via smooth $4$-manifolds. (See \cite{Kot} for an overview, and \cite{Kor} for a similar invariant defined via Lefschetz fibrations.) At least at the first glance, this invariant seems more computable than many others appeared in the literature before, yet to be seen how fine it is. We can summarize what we have shown so far:

\begin{prop}
Let $G$ be a finitely presented group. The only genus zero group is the trivial group. The only genus one groups are  $\Z$, $\Z_m$, $\Z \x \Z$, or $\Z \x \Z_m$. For the surface group $G = \pi_1(\Sigma_g)$, we have $g(G)=g$.  
\end{prop}

A natural problem within this context is:

\begin{problem} \label{groupgenera}
Determine $g(G)$ for a given finitely presented group. 
\end{problem}

\vspace{0.35in}
\noindent \textit{Acknowledgements.} We would like to thank Ron Fintushel for a helpful e-mail exchange.
The author was partially supported by the NSF grant DMS-0906912


\vspace{1in}
\begin{thebibliography}{99999}

\bibitem{ADK} D. Auroux, S. Donaldson, and L. Katzarkov, {\emph ``Singular Lefschetz pencils'',} Geom. 
Topol. 9 (2005), 1043--1114.

\bibitem{AK} S. Akbulut and C. Karakurt, {\emph ``Every 4-manifold is BLF'',} J. G{\" o}kova Geom. Topol. GGT 2 (2008) 
83--106. 

\bibitem{B0} R. I. Baykur, {\emph ``Symplectic structures, Lefschetz fibrations, and their generalizations on smooth four-manifolds'',} Thesis (Ph.D.)–Michigan State University. 2007. 136 pp. 

\bibitem{B1}  R. I. Baykur, {\emph ``Topology of broken Lefschetz fibrations and near-symplectic $4$-
manifolds''}, Pacific J. Math. 240 (2009), no.  2, 201--230.

\bibitem{B2} R. I. Baykur, {\emph ``Existence of broken Lefschetz fibrations'',} Int. Math. Res. Not. IMRN (2008), Art. ID rnn 101, 15 pp.

\bibitem{B3} R. I. Baykur, {\emph ``Handlebody argument for modifying achiral singularities'',} appendix to
\cite{L}), Geom. Topol. 13 (2009), 312--317.

\bibitem{BK} R. I. Baykur and S. Kamada, {\emph ``Classification of broken Lefschetz fibrations with small fiber genera'',} preprint, http://arxiv.org/abs/1010.5814.

\bibitem{SB} S. Behrens, {\emph ``On $4$-manifolds, folds and cusps'',} preprint.

\bibitem{C} K. L. Choi, {\emph ``Constructing a broken Lefschetz fibration of $S^4$ with a spun or twist-spun torus knot fiber'',} preprint, http://arxiv.org/abs/1107.1822. 

\bibitem{Donaldson} S. K. Donaldson, {\emph ``An application of gauge theory to four-dimensional topology'',} J. Differential Geom. 18 (1983), 279--315.

\bibitem{FS: knot surgery} R. Fintushel and R. J. Stern, {\emph ``Knots, links, and $4$-manifolds'',} Invent. Math. 134 (1998), no. 2, 363–-400.

\bibitem{FS: same SW} R. Fintushel and R. J. Stern, {\emph ``Families of simply connected $4$-manifolds with the same Seiberg-Witten invariants'',} Topology 43 (2004), no. 6, 1449-–1467.

\bibitem{Freedman} M. H. Freedman, {\emph ``The topology of four-dimensional manifolds'',} J. Differential Geom.
17 (1982), 357--453.


\bibitem{GK1} D. Gay and R. Kirby, {\emph ``Constructing Lefschetz-type fibrations on four-manifolds'',} Geom. Topol. 11 (2007), 2075--2115.

\bibitem{GK2} D. Gay and R. Kirby, {\emph ``Indefinite Morse $2$-functions; broken fibrations and generalizations'',} preprint, http://arxiv.org/abs/1102.0750.

\bibitem{GS} R.\,E. Gompf and A.\,I. Stipsicz, {\it 4-manifolds and Kirby calculus}, Graduate Studies in Math. {\bf 20}, Amer. Math. Soc., Providence, RI, 1999.

\bibitem{H1} K. Hayano, {\emph ``On genus-1 simplified broken Lefschetz fibrations'',} Algebraic \& Geometric Topology 11 (2011) 1267–1322.

\bibitem{H2} K. Hayano, {\emph ``A note on sections of broken Lefschetz fibrations'',} preprint, http://arxiv.org/abs/1104.1037.

\bibitem{HS1} K. Hayano and M. Sato, {\emph ``Four-manifolds admitting hyperelliptic broken Lefschetz fibrations'',} preprint, http://arxiv.org/abs/1110.0161.

\bibitem{HS2}  K. Hayano and M. Sato, {\emph  ``A signature formula for hyperelliptic broken Lefschetz fibrations'',} preprint, http://arxiv.org/abs/1110.5286.

\bibitem{Kor} M. Korkmaz, {\emph ``Lefschetz fibrations and an invariant of finitely presented groups,''} Int. Math. Res. Not. IMRN 2009, no. 9, 1547--1572. 

\bibitem{Kot} D. Kotschick, {\emph ``Four-manifold invariants of finitely presentable groups,''} Topology, geometry and field theory, 89--99, World Sci. Publ., River Edge, NJ, 1994. 

\bibitem{L} Y. Lekili, {\emph ``Wrinkled fibrations on near-symplectic manifolds'',} Geom. Topol. 13 (2009), 277--318.

\bibitem{LL} T.-J. Li and A. Liu, {\emph ``Symplectic structure on ruled surfaces and a generalized adjunction formula'',} Math. Res. Lett. 2 (1995), 453--471.

\bibitem{Lev0} H.\,I. Levine, {\emph Singularities of differentiable mappings,} Lecture Notes in Math. {\bf 192}, Springer (1971), 1--89.

\bibitem{MST} J. Morgan, Z. Szab\'o, C.~H. Taubes; {\emph ``A product formula for the Seiberg-Witten invariants and the generalized Thom conjecture'',} J. Differential Geom. 44 (1996), no. 4, 706–788. 


\bibitem{P} T. Perutz, {\emph ``Lagrangian matching invariants for fibred four-manifolds: I}, Geom. Topol. {\bf 11} (2007), 759--828.

\bibitem{Sa} O. Saeki, {\emph ``Elimination of definite fold'',} Kyushu J. Math. 60 (2006), 363--382.

\bibitem{W} J. Williams. {\emph ``The h-principle for broken Lefschetz fibrations'',} Geom. Topol.  14 (2010), no. 2, 1015--1063. 

\end{thebibliography}
\end{document}